\documentclass[12pt]{article}
\usepackage{graphicx}
\usepackage{amssymb,amsmath,amsthm}
\usepackage[latin1]{inputenc}
\usepackage[english]{babel}
\newtheorem{theorem}{Theorem}[section]
\newtheorem{lemma}[theorem]{Lemma}

\newtheorem*{claim}{Claim}
\theoremstyle{remark}\newtheorem*{remark}{Remark}
\theoremstyle{remark}

\newcommand{\la}{\lambda}
\DeclareMathOperator{\defi}{def}
\DeclareMathOperator{\order}{order}
\DeclareMathOperator{\stab}{Stab}

\begin{document}

\title{Asymptotically tight bounds on subset sums}
\author{Simon Griffiths}

\maketitle

\begin{abstract}

For a subset $A$ of a finite abelian group $G$ we define $\Sigma(A)=\{\sum_{a\in B}a:\,B\subset A\}$.  In the case that $\Sigma(A)$ has trivial stabiliser, one may deduce that the size of $\Sigma(A)$ is at least quadratic in $|A|$;  the bound $|\Sigma(A)|\ge |A|^{2}/64$ has recently been obtained by De Vos, Goddyn, Mohar and \v S\' amal \cite{dvgms}.  We improve this bound to the asymptotically best possible result $|\Sigma(A)|\ge (1/4-o(1))|A|^{2}$.  

We also study a related problem in which $A$ is any subset of $\mathbb{Z}_{n}$ with all elements of $A$ coprime to $n$; it has recently been shown, by Vu \cite{Vu}, that if such a set $A$ has the property $\Sigma(A)\neq \mathbb{Z}_{n}$ then $|A|=O(\sqrt{n})$.  This bound was improved to $|A|\le 8\sqrt{n}$ by De Vos, Goddyn, Mohar and \v S\' amal \cite{dvgms}, we further improve the bound to the asymptotically best possible result $|A|\le (2+o(1))\sqrt{n}$.

\end{abstract}

\section{Introduction}

For a subset $A$ of a finite abelian group $G$ we define, \begin{equation*} \Sigma(A)=\Big\{\sum_{a\in B}a:\, B\subset A\Big\}\end{equation*} the set of all elements which may be expressed as a sum of elements of $A$ (with repetition not allowed).  For a subset $S\subset G$ the stabiliser $\stab(S)$ of $S$ is the set of elements $g\in G$ such that $S+g=S$; the stabiliser is a subgroup of $G$.  We say that $S$ has trivial stabiliser if $\stab(S)=\{0\}$.  A recent result of De Vos, Goddyn, Mohar and \v S\' amal \cite{dvgms} shows that if $\Sigma(A)$ has trivial stabiliser then its size is at least quadratic in the size of $A$.

\begin{theorem}\label{dvgms} \cite{dvgms} Let $A\subset G\setminus \{0\}$, and suppose that $\Sigma (A)$ has trivial stabiliser, then $|\Sigma(A)|\ge |A|^{2}/64$\end{theorem}

It is noted in \cite{dvgms} that the above theorem can be proved with $1/64$ replaced by $1/48-o(1)$; where $o(1)$ denotes a function which converges to $0$ as $|A|\to \infty$.  Our aim is to improve this further by showing the result holds with the constant replaced by $1/4-o(1)$.  This result is asymptotically best possible as seen by considering $A=\{-n,-(n-1),...,n-1,n\}\subset \mathbb{Z}_{N}$, with $N$ large.

\begin{theorem}\label{ismain} Let $A\subset G\setminus \{0\}$, and suppose $\Sigma (A)$ has trivial stabiliser, then $|\Sigma(A)|\ge (1/4-o(1))|A|^{2}$\end{theorem}

We follow \cite{dvgms} in deducing related results for the case $A$ has non-trivial stabiliser.  

\begin{theorem}\label{wstab} Let $A\subset G$, and let $H$ be the stabiliser of $\Sigma (A)$ then, \begin{equation*} |\Sigma(A)|\ge (1/4-o(1))|A\setminus H|^{2}\end{equation*} where $o(1)$ denotes a function which converges to $0$ as $|A\setminus H|/|H|\to \infty$.\end{theorem}

\begin{remark} As the $o(1)$ term of the above statement converges to zero as $|A\setminus H|/|H|\to \infty$, rather than as $|A\setminus H|\to \infty$, our result is only an improvement on the result of \cite{dvgms} in the case that $|A\setminus H|$ is large relative to $|H|$.\end{remark}

Erd\H os and Heilbronn \cite{EH} proved that if $A\subset \mathbb{Z}_{p}$, the integers modulo $p$ (with $p$ prime), and $|A|\ge 3\sqrt{6p}\,$ then $\Sigma (A)=\mathbb{Z}_{p}\,$ -- the connection of this result to the current discussion is that this result is proved by giving a quadratic lower bound on $|\Sigma (A)|$ for $A\subset \mathbb{Z}_{p}$.  They conjectured that the constant $3\sqrt{6}$ of their theorem could be replaced by $2$, this was proved by Olson \cite{O} and further sharpened by Dias da Silva and Hamidoune \cite{DdSH}.  To prove a similar result in $\mathbb{Z}_{n}$, for $n$ composite, one must put extra conditions on the set $A$.  The following theorem of Vu \cite{Vu} shows one way in which this can be done is to demand that $A$ is contained in $\mathbb{Z}_{n}^{*}$, the elements coprime to $n$.

\begin{theorem} (Vu \cite{Vu}) There is a constant $c$ such that every subset $A\subset \mathbb{Z}_{n}^{*}\subset \mathbb{Z}_{n}$ with $|A|\ge c\sqrt{n}$ satisfies $\Sigma(A)=\mathbb{Z}_{n}$\end{theorem}

The constant obtained in the original proof of the theorem is quite large.  As a corollary of their main theorem De Vos, Goddyn, Mohar and \v S\' amal \cite{dvgms} improved the constant to $c=8$.  We improve this further, by replacing the constant by $2+o(1)$.  

\begin{theorem}\label{asymptoticolson} Let $A\subset \mathbb{Z}_{n}^{*}\subset \mathbb{Z}_{n}$ be such that $\Sigma(A)\neq\mathbb{Z}_{n}$, then $|A|\le (2+o(1))\sqrt{n}$\end{theorem}

\begin{remark} This result is asymptotically best possible, consider for example the case $n=p^{2}$, where $p$ is a large prime, and consider $A=\{-(p-1),...,-1,1,...,p-1\}$.  All elements of $A$ are coprime to $n$, and the element $p(p-1)/2+1$ is not in $\Sigma(A)$ so that $\Sigma(A)\neq \mathbb{Z}_{n}$, while $|A|=2p-2=(2-o(1))\sqrt{n}$.\end{remark}

The results Theorem \ref{ismain} and Theorem \ref{wstab} will be deduced from Theorem \ref{main} which we state below.  Theorem \ref{main} will state that if $A\subset G\setminus \{0\}$ is such that $\Sigma(A)$ has trivial stabiliser, and has the extra property that $A\cap (-A)=\phi$ then $\Sigma(A)\ge (1/2-o(1))|A|^{2}$.  We now state this result formally.

Let $n_{9}=1$, and for $k\ge 10$ let $n_{k}=2^{k^{2}}$.  We define also a sequence of real numbers $\alpha_{k}$ as follows: $\alpha_{9}=1/64$ and for $k\ge 10$ we define\begin{equation*} \alpha_{k}=\min\Big\{\,\frac{6}{5}\alpha_{k-1}\, ,\frac{1}{2}-\frac{1}{2^{k-1}}\Big\}\end{equation*} It is clear that the sequence $\alpha_{k}$ is increasing and converges to $1/2$ as $k\to \infty$.

\begin{theorem}\label{main} Let $A\subset G\setminus\{0\}$ and let $k\ge 9$.  Suppose that $A\cap (-A)=\phi$, that $\Sigma(A)$ has trivial stabiliser and that $|A|\ge n_{k}$, then $|\Sigma(A)|\ge \alpha_{k}|A|^{2}$\end{theorem}

The layout of the article is as follows.  We shall conclude the introduction by introducing definitions and observations that will be used throughout the article.  In Section \ref{pp} we assume Theorem \ref{main} and deduce Theorems \ref{ismain} and \ref{wstab}.  In Section \ref{prel} we establish preliminary lemmas required for the proof of Theorem \ref{main}, we then prove Theorem \ref{main} in Section \ref{pom}.  We then turn to Theorem \ref{asymptoticolson}, which is proved in Section \ref{poao}.  

Given a set $S\subset G$ and an element $c\in G$ we write $S+c$ for the set $\{s+c:s\in S\}$.  We define $\la_{S}(c)$, or simply $\la(c)$, to be the number of elements in $S+c$ which are not in $S$, ie. $\la(c)=|(S+c)\setminus S|$.  The following properties of $\la$ are elementary.

\noindent (i) $\la(0)=0$ \\ (ii) $\la(-c)=\la(c) \!\qquad\qquad \qquad \text{for all}\,\,c\in G$ \\ (iii)$\la(b+c)\le \la(b)+\la(c) \qquad \text{for all}\, \,b,c\in G$ \qquad (subadditivity of $\la$)\vspace{0.1cm}

For sets $A_{1},...,A_{r}\subset G$, we define their sumset, \begin{equation*}\Sigma_{i=1}^{r}A_{i}=\{a_{1}+...+a_{r}:\, a_{i}\in A_{i}\, \,\text{for }\,i=1,...,r \}\end{equation*} We shall also use the notation $rA$ to denote $\sum_{i=1}^{r}A=\{a_{1}+...+a_{r}:a_{i}\in A \}$.  The key sumset inequality we shall use is Kneser's addition theorem.

\begin{theorem} (Kneser \cite{K}) Let $A_{1},...,A_{r}\subset G$ and let $H$ be the stabiliser of $\sum_{i=1}^{r}A_{i}$, then, \begin{equation*}\big|\sum_{i=1}^{r}A_{i}\big|\ge \sum_{i=1}^{r}|A_{i}|-(r-1)|H|\end{equation*}\end{theorem}

Let $B\subset A\subset G$.  We may express $\Sigma(A)$ as $\Sigma(A)= \Sigma(B)+\Sigma(A\setminus B)$.  It is then easily observed that,\begin{equation*} \stab(\Sigma(B))\subset \stab(\Sigma(A))\end{equation*} In particular if $\Sigma(A)$ has trivial stabiliser then so does $\Sigma(B)$.  These observations shall be used throughout.

\section{Proofs of Theorem \ref{ismain} and Theorem \ref{wstab}}\label{pp}

In this section we assume Theorem \ref{main} and prove Theorem \ref{ismain} and Theorem \ref{wstab}.  Let $(n_{k})_{k\ge 9}$ and $(\alpha_{k})_{k\ge 9}$ be the sequences defined above, for a natural number $n$ let $k(n)$ be the largest integer $k$ such that $n\ge 2n_{k}$.  We define, \begin{equation*} f(n)=\alpha_{k(n)}/2-1/n^{2}\end{equation*} It is clear that $f(n)$ satisfies $f(n)\to 1/4$ as $n\to \infty$, ie. $f$ is of the form $1/4-o(1)$.

\begin{proof}[Proof of Theorem \ref{ismain}] Let $A\subset G\setminus \{0\}$ be such that $\Sigma (A)$ has trivial stabiliser, let $n=|A|$.  We prove Theorem \ref{ismain} by showing $|\Sigma(A)|\ge f(n)n^{2}$.  Let $m=\lfloor n/2\rfloor$, it is possible to partition $A$ into two subsets $A_{1},A_{2}$ with cardinalities $m$ and $n-m$ respectively, with the property $A_{i}\cap (-A_{i})=\phi$ for $i=1,2$.  Let $k=k(n)$, by the definition of $k(n)$ we have that $m\ge n_{k}$, and so applying Theorem \ref{main} to $A_{1},A_{2}$ individually we obtain, $|\Sigma(A_{1})|\ge \alpha_{k}m^{2}$ and $|\Sigma(A_{2})|\ge \alpha_{k}(n-m)^{2}$.  We have that, $\Sigma(A)=\Sigma(A_{1})+\Sigma(A_{2})$, and so applying Kneser's theorem (and using the fact that $\Sigma(A)$ has trivial stabiliser) we obtain,\begin{equation*} |\Sigma(A)|\ge |\Sigma(A_{1})|+|\Sigma(A_{2})|-1\ge \alpha_{k} m^{2}+\alpha_{k} (n-m)^{2}-1\end{equation*} We note $m^{2}+(n-m)^{2}\ge n^{2}/2$ and so, \begin{equation*} |\Sigma(A)|\ge \frac{\alpha_{k}n^{2}}{2}-1=f(n)n^{2}\vspace{-0.6cm}\end{equation*}\end{proof} \vspace{0.1cm}
  
Let $A\subset G\setminus \{0\}$, and suppose $\Sigma(A)$ has trivial stabiliser, Theorem \ref{ismain} gives that $|\Sigma(A)|\ge f(|A|)|A|^{2}$ for some function $f(n)\to 1/4$ as $n\to \infty$.  Of course by setting $f'(n)=f(n)-1/n^{2}$, we have \begin{equation}\label{fprime} |\Sigma(A)|\ge 1+f'(|A|)|A|^{2}\end{equation} for a function $f'(n)\to 1/4$ as $n\to \infty$.  Let $m(n)=\lfloor \sqrt{n}\rfloor$, we define, \begin{equation*}f_{2}(n)=(1-m(n)^{-1})^{2}f'(m(n))\end{equation*}  It is clear that $f_{2}(n)\to 1/4$ as $n\to \infty$, ie. $f_{2}(n)$ is of the form $1/4-o(1)$

\begin{proof}[Proof of Theorem \ref{wstab}] Let $m(n)$ and $f_{2}(n)$ be as defined as above.  Let $A\subset G$, let $H$ be the stabiliser of $\Sigma(A)$, and let $n=|A\setminus H|/|H|$.  We prove Theorem \ref{wstab} by demonstrating that, \begin{equation*}|\Sigma(A)|\ge f_{2}(n)|A\setminus H|^{2}\end{equation*}  We work in $G/H$, an element $Q$ of $G/H$ is a coset of $H$.  We let $h=|H|$ and define a sequence $A_{1},...,A_{h}$ of subsets of $G/H$ by \begin{equation*} A_{i}=\{Q\in G/H:|A\cap Q|\ge i\} \end{equation*} Note that the coset $Q$ appears in exactly $|A\cap Q|$ of the sets $A_{i}$, this implies, \begin{equation*} \sum_{i=1}^{h}|A_{i}|=|A|\end{equation*} \vspace{-0.2cm} Writing $A'_{i}$ for $A_{i}\setminus \{H\}$ we have, \begin{equation*} \sum_{i=1}^{h}|A'_{i}|=|A\setminus H|\end{equation*} The key observation is that $\Sigma(A)$ consists exactly of those elements which lie in cosets in $\Sigma(A'_{1})+...+\Sigma(A'_{h})$.  So that,\begin{equation*} |\Sigma(A)|= h|\Sigma(A'_{1})+...+\Sigma(A'_{h})| \end{equation*} We now put a lower bound on $|\Sigma(A'_{1})+...+\Sigma(A'_{h})|$.  The sets $A'_{1},...,A'_{h}$ are decreasing in size, let $j$ be maximal such that $|A'_{j}|\ge m(n)$.  Now $\sum_{i> j}|A'_{i}|<h m(n)$ so that we have, \begin{equation*} \Sigma_{i=1}^{j}|A'_{i}|\ge |A\setminus H|-hm(n)\ge h(n-m(n))\ge hn(1-m(n)^{-1})\end{equation*} For all $i\le j$ we have $|A'_{i}|\ge m(n)$ so that from (\ref{fprime}), and the fact that $f'$ is increasing, we have, \begin{equation*} |\Sigma(A'_{i})|\ge 1+f'(|A'_{i}|)|A'_{i}|^{2}\ge 1+f'(m(n))|A'_{i}|^{2}\end{equation*} Since $H$ is the stabiliser of $\Sigma(A)$ we must have that $\Sigma(A'_{1})+...+\Sigma(A'_{j})$ has trivial stabiliser in $G/H$, and so by Kneser's theorem, $|\Sigma(A'_{1})+...+\Sigma(A'_{j})|\ge |\Sigma(A'_{1})|+...+|\Sigma(A'_{j})|-(j-1)$, it follows that, \begin{equation*} |\Sigma(A'_{1})+...+\Sigma(A'_{j})|\ge \sum_{i=1}^{j}\big( 1+f'(m(n))|A'_{i}|^{2}\big) -(j-1)\ge f'(m(n))\sum_{i=1}^{j}|A'_{i}|^{2}\end{equation*} We now apply Cauchy-Schwarz to obtain, \begin{equation*} |\Sigma(A'_{1})+...+\Sigma(A'_{j})|\ge \frac{f'(m(n))}{j}\Big(\sum_{i=1}^{j}|A'_{i}|\Big)^{2}\ge \frac{f'(m(n))}{h}\big(hn(1-m(n)^{-1})\big)^{2}\end{equation*} and so,  \begin{equation*} |\Sigma(A'_{1})+...+\Sigma(A'_{j})|\ge (1-m(n)^{-1})^{2}f'(m(n))hn^{2}\end{equation*} We now simply note that $|\Sigma(A'_{1})+...+\Sigma(A'_{h})|\ge |\Sigma(A'_{1})+...+\Sigma(A'_{j})|$ and recall that, $|\Sigma(A)|= h|\Sigma(A'_{1})+...+\Sigma(A'_{h})|$ to obtain, \begin{equation*} |\Sigma(A)|\ge (1-m(n)^{-1})^{2}f'(m(n))h^{2}n^{2}= f_{2}(n)|A\setminus H|^{2} \vspace{-0.6cm}\end{equation*}\end{proof}

\section{Some preliminary lemmas}\label{prel}

In this section we prove the key lemmas which are central to our proof of Theorem \ref{main} in Section \ref{pom}.  That proof will work by building up a set $B\subset A$ with $|\Sigma (B)|$ large.  During this process we shall inspect the current choice of $B$ and let $S=\Sigma(B)$, we shall then attempt to find an element $c\in C=A\setminus B$ with $\la_{S}(c)$ relatively large, and then add $c$ to $B$, this leads to set with size one more than $|B|$, with a significantly larger set of sums.  We shall recall results proved previously (Lemma \ref{ssmall} \cite{EH} and Lemma \ref{slarge} \cite{dvgms}) which enable one to put a lower bound on $\max_{c\in C}\la(c)$.  More importantly we strengthen Lemma \ref{slarge} to Lemma \ref{thegoodincrease}, a result which is best possible, up to an error term; furthermore for our applications of the lemma the error term is small asymptotically.  

For a subset $Q\subset G$, we define $\defi_{Q}(S)=\min\{|S\cap Q|,|Q\setminus S|\}$.   Let $\rho(d)$ be the number of representations of $d$ as a difference of two elements of $S$, ie. $\rho(d)=|\{(x,y)\in S^{2}:x-y=d\}|=|S\cap (S+d)|$.  The equality $|(S+d)\cap S|+|(S+d)\setminus S|=|S|$ implies that $\la (d)=|S|-\rho(d)$.  A key fact will be that, \begin{equation*} \sum_{d\in G}\rho (d)= |S|^{2}\end{equation*} this follows from the fact that each pair of elements $(x,y)$ of $S$ have a unique difference $x-y$, so is counted exactly once by the sum.  In the following lemmas, $H$ denotes a finite abelian group.  The following two lemmas were originally proved in \cite{EH} and \cite{dvgms} respectively, we give proofs so that the reader may become familiar with ideas which we shall use in the more involved proof of Lemma \ref{thegoodincrease}.

\begin{lemma}\label{ssmall}\cite{EH} Let $C,S\subset H$ be such that $\defi_{H}(S)\le |C|/2$, then \begin{equation*} \frac{1}{|C|}\sum_{c\in C}\la (c)\ge \frac{\defi_{H}(S)}{2}\end{equation*} In particular $\la (c)\ge \defi_{H}(S)/2$ for some $c\in C$.\end{lemma}

\begin{proof} Exchanging $S$ for $H\setminus S$ if required we may assume $\defi_{H}(S)=|S|\le |H|/2$, this is allowed as $\la_{S}(c)=\la_{H\setminus S}(c)$ for all $c\in C$.  From the facts $\la(c)=|S|-\rho(c)$ and $\sum_{d\in H}\rho(d)=|S|^{2}$ we deduce that, \begin{equation*}\sum_{c\in C}\la(c) = |C||S|-\sum_{c\in C} \rho(c)\ge |C||S|-|S|^{2}\end{equation*} We have also that $|C|\ge 2\defi_{H}(S)=2|S|$ so that $|C|-|S|\ge |C|/2$, it follows that, \begin{equation*}\sum_{c\in C}\la(c) \ge |S|(|C|-|S|)\ge \,\frac{|C||S|}{2}=\,\frac{|C|\defi_{H}(S)}{2}\vspace{-0.6cm}\end{equation*}\end{proof}\vspace{0.1cm}

\begin{lemma}\label{slarge}\cite{dvgms}  Let $C,S\subset H$ be such that $H=\langle C\rangle$ and $\defi_{H}(S)\ge |C|/2$ then there exists $c\in C$ with $\la (c)\ge |C|/8$\end{lemma}

\begin{proof}[Sketch proof] Again we may assume $\defi_{H}(S)=|S|\le |H|/2$.  We cannot proceed as in the previous proof, because we obtain no information in the case $|C|\le |S|$.  Instead we consider not only the elements of $C$ but all elements which can be expressed as a sum of a small number of elements of $C$, we then obtain the required result from the subadditivity of $\la$.  Suppose for contradiction that $\la(c)<|C|/8$ for all $c\in C$, and suppose $d$ may be expressed as a sum of $k$ elements of $C$, then by the subadditivity of $\la$ we have $\la(d)<k|C|/8$.  Using the notation $rA$ for the set $\{a_{1}+...+a_{r}:a_{i}\in A\}$, we have that the set of elements that may be expressed as a sum of at most $r$ elements of $C$ is the set $rC^{*}$ where $C^{*}=C\cup \{0\}$.  We may assume for all $d\in rC^{*}$ that $\la(d)< r|C|/8$ and so $\rho(d)>|S|-r|C|/8$.  An application of Kneser (see \cite{dvgms} for details) now shows that for $r=\lfloor 4|S|/|C|\rfloor$ we have $|rC^{*}|\ge 2|S|$.  For this value of $r$ we have $\rho(d)>|S|-r|C|/8\ge |S|-|S|/2=|S|/2$ for all $d\in rC^{*}$, so that, \begin{equation*} \sum_{d\in rC^{*}}\rho (d)> |rC^{*}| \frac{|S|}{2}\ge |S|^{2}\end{equation*} This contradicts the equality $\sum_{d\in H}\rho(d)= |S|^{2}$, and the result is proved.\end{proof}

With greater care, and with extra conditions, the result of Lemma \ref{slarge} can be greatly improved; one can find an element $c\in C$ with $\la (c)$ almost $|C|/2$.  The extra conditions we need are that $C\cap (-C)=\phi$, that $|C|$ is large, and that $\defi_{H}(S)/|C|$ is large.  We will again use the subadditivity of $\la$; in our sketch proof of Lemma \ref{slarge} we used that if $\la(c)< \eta$ then for each $r$ we have $\la(d)<r\eta$ for all $d\in rC^{*}$ and so $\rho(d)>|S|-r\eta$ for all $d\in rC^{*}$.  To obtain an the improved result we must use this statement simultaneously for many values of $r$ rather than once for a single value of $r$.  Again it is important to have a lower bound on $|rC^{*}|$ (although in the sequel we take a different set $C^{*}$), we first prove Lemma \ref{rc} which gives an improved lower bound on $|rC^{*}|$, we will then prove Lemma \ref{thegoodincrease} which gives a new bound on $\la(c)$.  We define $C^{*}=C\cup (-C)\cup \{0\}$.

\begin{lemma}\label{rc} Let $C\subset H$ be such that $H=\langle C\rangle$, $C\cap (-C)=\phi$,  $\Sigma(C)$ has trivial stabiliser and $|C|\ge 2^{2k+11}$ then for all positive integers $r$ we have that either $rC^{*}=H$ or \begin{equation*} |rC^{*}|\ge 2\Big(1-\frac{1}{2^{k+2}}\Big)r|C| \end{equation*}\end{lemma}

\begin{proof} Let $K$ be the stabiliser of $rC^{*}$.  If $C\subset K$ then $K=H$ and $rC^{*}=H$ and we are done.  So we may assume $C^{*}$ meets at least one non-trivial coset of $K$, note that $C^{*}$ also meets $K$ as $0\in C^{*}\cap K$.  Let $Q_{1},...,Q_{p}$ be the non-trivial cosets of $K$ meeting $C^{*}$, it follow that $rC^{*}$ is the set of elements of the cosets in $r\{K,Q_{1},...,Q_{p}\}$.  By the definition of $K$ the set $r\{K,Q_{1},...,Q_{p}\}$ has trivial stabiliser in $H/K$ so that an application of Kneser's theorem yields, \begin{equation*} |r\{K,Q_{1},...,Q_{p}\}|\ge r(p+1)-(r-1)\ge rp\end{equation*} and so $|rC^{*}|\ge rp|K|$.  So to complete the proof we must show that,\begin{equation*} |K|\ge 2\Big(1-\frac{1}{2^{k+2}}\Big)\frac{|C|}{p} \end{equation*}  In the case that $|C\cap K|\ge |C|/2^{k+2}$ we use the fact that $K$ contains $\Sigma(C\cap K)$ together with Theorem \ref{dvgms} (which we may apply to $C\cap K$, as $\Sigma(C\cap K)$ has trivial stabiliser) to deduce, \begin{equation*} |K|\ge |\Sigma(C\cap K)|\ge \frac{|C\cap K|^{2}}{64}\ge \frac{1}{64}\frac{|C|^{2}}{2^{2k+4}}\ge 2|C|\end{equation*} the final inequality follows from $|C|\ge 2^{2k+11}$, and we are done.  We may now assume that $|C\cap K|\le |C|/2^{k+2}$, this implies $|C\setminus K|\ge (1-2^{-(k+2)})|C|$ and likewise $|C'\setminus K|\ge (1-2^{-(k+2)})|C'|$, where $C'$ denotes $C\cup (-C)$.  Alternatively $|C' \cap \bigcup_{i=1}^{p} Q_{i}|\ge (1-2^{-(k+2)})|C'|$, so by the pigeon hole principle some coset $Q$ must have, \begin{equation*} |C'\cap Q|\ge \Big(1-\frac{1}{2^{k+2}}\Big)\frac{|C'|}{p}= 2 \Big(1-\frac{1}{2^{k+2}}\Big)\frac{|C|}{p}\end{equation*} and so noting $|K|=|Q|\ge |C'\cap Q|$ we are done.\end{proof}

We now deduce the required strengthening of Lemma \ref{slarge},

\begin{lemma}\label{thegoodincrease} Let $C\subset H$ be such that $H=\langle C\rangle$, $C\cap (-C)=\phi$, $\Sigma(C)$ has trivial stabiliser and $|C|\ge 2^{2k+11}$, and let $S\subset H$ be such that $\defi_{H}(S)\ge 2^{k+1}|C|$ then there exists $c\in C$ with $\la (c)\ge (1-2^{1-k})|C|$\end{lemma}

\begin{proof} For all $d\in H$ we have $\la_{S}(d)=\la_{H\setminus S}(d)$, so that exchanging $S$ with $H\setminus S$ if required, we may assume $\defi_{H}(S)=|S|\le |H|/2$.  As in previous proofs we shall proceed by assuming the result fails and from this deduce that $\sum_{d} \rho(d)> |S|^{2}$, a contradiction.  On this occasion we must be much more precise in our lower bound on $\sum_{d}\rho(d)$.  It will be useful to use the fact that $\sum_{d}\rho(d)=\int_{0}^{|S|}|D_{t}| \, dt $ where $D_{t}=\{d\in H:\rho(d)\ge t\}$.
\footnote{We choose this representation, rather than the sum $\sum_{t=1}^{|S|}|D_{t}|$, as it means we do not have to concern ourselves with integer parts, etc.} 
\\ \indent Suppose that every element $c\in C$ has $\la(c)<(1-2^{1-k})|C|$, this implies $\la(c)<(1-2^{1-k})|C|$ for all $c\in C^{*}$, since $\la(0)=0$ and $\la(-c)=\la(c)$.  So that by the subadditivity of $\la$ we have for all $r$ and for all $d\in rC^{*}$ that $\la (d)<r(1-2^{1-k})|C|$, and so $\rho(d)>|S|-r(1-2^{1-k})|C|$.  For each $r=1,...,\lfloor |S|/|C|\rfloor$, we define, \begin{equation*} I_{r}=(|S|-(r+1)(1-2^{1-k})|C|,|S|-r(1-2^{1-k})|C|]\end{equation*} For $t\in I_{r}$ we have $t\le |S|-r(1-2^{1-k})|C|$ and so $D_{t}\supset rC^{*}$, applying Lemma \ref{rc} we deduce, $|D_{t}|\ge 2(1-2^{-(k+2)})r|C|$.
\\ \indent Note also that for $t\le |S|-(\lfloor |S|/|C| \rfloor+1) (1-2^{1-k})|C|$ we have $D_{t}\supset \lfloor |S|/|C|\rfloor C^{*}$ and so $|D_{t}|\ge 2(1-2^{-(k+2)}) \lfloor |S|/|C|\rfloor |C|$.  Using the fact that $\lfloor |S|/|C|\rfloor\ge (|S|-|C|)/|C|\ge(1-2^{-(k+1)})|S|/|C|$ we have, \begin{equation*} |D_{t}|\ge 2\Big(1-\frac{1}{2^{k+2}}\Big)\Big\lfloor \frac{|S|}{|C|}\Big\rfloor |C|\ge 2\Big(1-\frac{1}{2^{k+2}}\Big) \Big(1-\frac{1}{2^{k+1}}\Big)|S| \end{equation*} 
for all $t\le |S|-(\lfloor |S|/|C| \rfloor +1) (1-2^{1-k})|C|$.  Note that $\lfloor |S|/|C|\rfloor +1\le |S|/|C|+1\le (1+2^{-(k+1)})|S|/|C|$ so that the above bound on $|D_{t}|$ hold for all $t$ satisfying $t\le |S|-(1-2^{1-k})(1+2^{-(k+1)})|S|$ and so certainly for all $t\le 3|S|/2^{k+1}$.  We obtain, \begin{equation*}  \sum_{d\in H}\rho(d)= \int_{0}^{|S|}|D_{t}| \, dt\ge \sum_{r=1}^{\lfloor |S|/|C|\rfloor}\int_{t\in I_{r}}|D_{t}| \, dt+\int_{0}^{3|S|/2^{k+1}}|D_{t}| \, dt \end{equation*} 
and so, \begin{equation*} \sum_{d\in H}\rho(d) \ge \sum_{r=1}^{\lfloor |S|/|C|\rfloor } 2\Big(1-\frac{1}{2^{k-1}} \Big) \Big(1-\frac{1}{2^{k+2}}\Big)r|C|^{2}+\frac{3|S|}{2^{k+1}}2\Big(1-\frac{1}{2^{k+2}}\Big) \Big(1-\frac{1}{2^{k+1}}\Big)|S|\end{equation*}  
Since $\lfloor |S|/|C|\rfloor \ge (1-2^{-(k+1)})|S|/|C|$ we have that $\sum_{r=1}^{\lfloor |S|/|C|\rfloor }r= (\lfloor|S|/|C|\rfloor)(\lfloor|S|/|C|\rfloor+1)/2 \ge (1-2^{-(k+1)})|S|^{2}/2|C|^{2}$, and so, \begin{equation*}  \sum_{d\in H}\rho(d)\ge \Big(1-\frac{1}{2^{k+1}}\Big)  \Big(1-\frac{1}{2^{k-1}}\Big) \Big(1-\frac{1}{2^{k+2}}\Big)|S|^{2} + \frac{3}{2^{k}}\Big(1-\frac{1}{2^{k+2}}\Big) \Big(1-\frac{1}{2^{k+1}}\Big)|S|^{2}\end{equation*}  However this quantity is larger than $|S|^{2}$, a contradiction. \end{proof}

\section{Proof of Theorem \ref{main}}\label{pom}

Let us recall the sequences $(n_{k})_{k\ge 9}$ and $(\alpha_{k})_{k\ge 9}$ which appear in the statement of Theorem \ref{main}. The sequence $(n_{k})_{k\ge 9}$ is defined by $n_{9}=1$, and for $k\ge 10$ by $n_{k}=2^{k^{2}}$.  The only information we shall use about this sequence is that for $k\ge 10$ it satisfies, \begin{equation} n_{k}> 2^{5k+15} \qquad \text{and}\qquad \frac{n_{k}}{2^{k}}\ge n_{k-1} \end{equation}  We recall the sequence $(\alpha_{k})_{k\ge 9}$ defined by, $\alpha_{9}=1/64$ and for $k\ge 10$ by, \begin{equation*} \alpha_{k}=\min\Big\{\,\frac{6}{5}\alpha_{k-1}\, ,\frac{1}{2}-\frac{1}{2^{k-1}}\Big\}\end{equation*}

Having proved all the necessary preliminary results we now proceed towards our proof of Theorem \ref{main}.  Our proof is by induction on $k$, the case $k=9$ follows from Theorem \ref{dvgms}, so we turn to the induction step, we let $k\ge 10$ and suppose the result holds for all smaller values of $k$.  Let $A$ be a set of size $n$ satisfying the conditions of Theorem \ref{main}, we may assume $n\ge n_{k}$, else there is nothing to prove.  As mentioned previously we shall build up a set $B\subset A$ with $\Sigma (B)$ large.  We define a function $g(t)$ for $t\in\{\lfloor n/2^{k}\rfloor,...,\lceil n-n/2^{k}\rceil \}$ by $g(\lfloor n/2^{k}\rfloor) =0$ and then inductively by,\begin{equation*} g(t+1)=g(t)+\Big(1-\frac{1}{2^{k-1}}\Big)(n-t)\end{equation*} We shall deduce Theorem \ref{main} from the following lemma.

\begin{lemma}\label{mainlemma} Let $A$ be as above then for all $t\in\{\lfloor n/2^{k}\rfloor,...,\lceil n-n/2^{k} \rceil\}$ there is a subset $B\subset A$ of cardinality $t$ with either $|\Sigma(B)|\ge g(t)$ or $|\Sigma(B)|\ge \alpha_{k}n^{2}$.\end{lemma}

Theorem \ref{main} now follows by taking $B$ of cardinality $\lceil n-n/2^{k}\rceil$ with either $|\Sigma(B)|\ge g(\lceil n-n/2^{k}\rceil)$ or $|\Sigma(B)|\ge \alpha_{k}n^{2}$, in the latter case we are done immediately as $|\Sigma(A)|\ge |\Sigma(B)|$.  In the former case we note that, \begin{equation*}g(\lceil n-n/2^{k}\rceil)\ge \Big(1-\frac{1}{2^{k-1}}\Big)\sum_{t=n/2^{k}}^{n-n/2^{k}}n-t= \Big(1-\frac{1}{2^{k-1}}\Big)\sum_{t=n/2^{k}}^{n-n/2^{k}}t \end{equation*} and we may bound the sum by,\begin{equation*} \sum_{t=n/2^{k}}^{n-n/2^{k}}t  \ge \Big(1-\frac{1}{2^{k}}\Big)^{2}\frac{n^{2}}{2}-\frac{n^{2}}{2^{2k+1}} \ge \Big(\frac{1}{2}-\frac{1}{2^{k}}\Big)n^{2}\end{equation*} and so, \begin{equation*}|\Sigma(A)| \ge \Big(1-\frac{1}{2^{k-1}}\Big)\Big(\frac{1}{2}-\frac{1}{2^{k}}\Big)n^{2}\ge \Big(\frac{1}{2}-\frac{1}{2^{k-1}}\Big)n^{2}\ge \alpha_{k} n^{2}\end{equation*} and Theorem \ref{main} is proved.

We must now prove Lemma \ref{mainlemma}.  We do this by induction on $t$.  The result is trivial for $t=\lfloor n/2^{k}\rfloor$.  If we ever find a set $B$ with $|\Sigma(B)|\ge \alpha_{k}n^{2}$ then the induction is trivial from that point on.  So for the induction step we let $B$ be a set of size $t$ with $|\Sigma (B)|\ge g(t)$, and set $S=\Sigma (B)$ and $C=A\setminus B$, we will then show either that $|S|\ge \alpha_{k}n^{2}$ or that, \begin{equation} \la(c)\ge \Big(1-\frac{1}{2^{k-1}}\Big)|C|=\Big(1-\frac{1}{2^{k-1}}\Big)(n-t) \qquad \text{for some}\,\, c\in C\end{equation} the induction step is then completed by considering $B\cup \{c\}$, as $B\cup\{c\}$ is then a set of size $t+1$ with $\Sigma(B\cup\{c\})=S\cup (S+c)$ and so, \begin{equation*}|\Sigma(B\cup\{c\})|\ge |S|+\la(c)\ge g(t)+(1-2^{1-k})(n-t)=g(t+1)\end{equation*}

Let us first note a lower bound on $|S|$ which we shall use during the proof, since $|B|=t\ge \lfloor n/2^{k}\rfloor\ge n/2^{k+1}$, it is immediate from Theorem \ref{dvgms} that, \begin{equation}\label{Sbig} |S|\ge \frac{|B|^{2}}{64}\ge \frac{n^{2}}{2^{2k+8}}\ge 2^{3k+7}n \end{equation}

Let $H=\langle C\rangle$.  As in \cite{dvgms}, we must analyse the intersection of $S$ with the cosets of $H$.  We begin with some lower bounds on $|H|$ that we shall use during the proof.  The set $C$ is contained in $A$ so that $C\cap (-C)=\phi$ and $0\not \in C$, also this implies that $\Sigma(C)$ has trivial stabiliser, note also that $|C|=n-t\ge n/2^{k}$.  An application of Theorem \ref{dvgms} yields,
\begin{equation*} |H|\ge |\Sigma(C)|\ge \frac{|C|^{2}}{64}\ge \frac{n^{2}}{2^{2k+6}}\ge 2^{3k+9}n \end{equation*} while an application of the induction hypothesis of Theorem \ref{main} yields, \begin{equation*} |H|\ge |\Sigma(C)|\ge \alpha_{k-1}|C|^{2}\ge \frac{3}{4}\frac{\alpha_{k}n^{2}}{2^{2k}}\end{equation*}  We say that a coset $Q$ of $H$ is \emph{sparse} if $|S\cap Q|\le 2^{k+1}n$ and $Q$ is \emph{very sparse} if $|S\cap Q|\le |C|/2$, while $Q$ is \emph{dense} if $|Q\setminus S|\le 2^{k+1}n$, we note that if $Q$ is dense then $|S\cap Q|$ must be fairly large, in particular, \begin{equation}\label{SQbig} |S\cap Q|\ge |H|-2^{k+1}n\ge \frac{3}{4}\alpha_{k} n^{2}/2^{2k} -2^{k+1}n\ge \alpha_{k} n^{2}/2^{2k+1}\end{equation} as $\alpha_{k}n/2^{2k+2}\ge 2^{k+1}$.  

\begin{lemma}\label{la} If any of the following conditions (i),(ii) or (iii) hold then there is an element $c\in C$ with $\la(c)\ge (1-2^{1-k})|C|$. \\ (i) There is a coset $Q$ which is neither sparse nor dense.\\ (ii) There are at least $2^{2k+5}$ cosets that are sparse but not very sparse \\ (iii) No coset $Q$ is dense.\end{lemma}

\begin{proof} (i) Suppose $Q$ is neither sparse nor dense, then we have $\defi_{Q}(S)\ge 2^{k+1}n\ge 2^{k+1}|C|$, an application of Lemma \ref{thegoodincrease} to a shift of $S\cap Q$ then shows that there exists $c\in C$ with $\la_{S\cap Q}(c)\ge (1-2^{1-k})|C|$, and we are done, as $\la_{S}(c)\ge \la_{S\cap Q}(c)$. 
\\ \indent (ii) Suppose that $Q_{1},...,Q_{2^{2k+5}}$ are sparse but not very sparse, for each $i=1,...,2^{2k+5}$ let $S_{i}=S\cap Q_{i}$ and let $\bar{S}=\bigcup_{i=1}^{2^{2k+5}}S_{i}$, then, \begin{equation*}\la(c)=\la_{S}(c)\ge \la_{\bar{S}}(c)=\sum_{i=1}^{2^{2k+5}}\la_{S_{i}}(c) \end{equation*}   We write $\la_{i}(c)$ for $\la_{S_{i}}(c)$ and $\bar{\la}(c)$ for $\la_{\bar{S}}(c)$, we shall show that,\begin{equation*} \bar{\la}(c)= \sum_{i=1}^{2^{2k+5}}\la_{i}(c)\ge |C|\qquad \text{for some}\,\, c\in C\end{equation*} We first note that $|\bar{S}|\ge 2^{2k+5}|C|/2\ge 2^{2k+4}|C|$.  We work as in the proofs of results in Section \ref{prel}.   Suppose for contradiction that $\bar{\la}(c)<|C|$ for all $c\in C$, then by the subadditivity of $\bar{\la}$ (which follows from the subadditivity of the $\la_{i}$), we have that $\bar{\la}(d)<2^{2k+3}|C|$ for all $d\in 2^{2k+3}(C\cup \{0\})$.  For each $i$ we let $\rho_{i}(d)=|(S_{i}+d)\cap S_{i}|$, so that, $\rho_{i}(d)=|S_{i}|-\la_{i}(d)$ and, $\sum_{d\in G} \la_{i}(d)=|S_{i}|^{2}$.  We let $\bar{\rho}(d)=\sum_{i}\rho_{i}(d)$, it follows that, $\bar{\rho}(d)=|\bar{S}|-\bar{\la}(d)$.  So that, \begin{equation*} \bar{\rho}(d)>|\bar{S}|- 2^{2k+3}|C|\ge |\bar{S}|/2\qquad\text{for all}\,\,\, d\in 2^{2k+3}(C\cup \{0\}) \end{equation*} the final inequality follows from the bound $|\bar{S}|\ge 2^{2k+4}|C|$ obtained earlier. \\ \indent We shall reach a contradiction by putting a lower bound on $|2^{2k+3}(C\cup\{0\})|$, and using this to deduce that $\sum_{d\in G}\rho_{i}(d)>|S_{i}|^{2}$ for some $i$, a contradiction.  We show that $|2^{2k+3}(C\cup\{0\})|\ge 2^{k+2}n$.  Let $K$ be the stabiliser of $2^{2k+3}(C\cup\{0\})$, if $C\subset K$ then $K=H$ and $|K|= |H|\ge 2^{k+2}n$.  Hence we may assume $C\cup \{0\}$ meets a non-trivial coset of $K$, let $R_{1},...,R_{p}$ be the non-trivial cosets of $K$ which have non-empty intersection with $C$, then the elements of $2^{2k+3}(C\cup \{0\})$ are exactly the elements of the cosets in $2^{2k+3}\{K,R_{1},...,R_{p}\}$.  The definition of $K$ implies that $2^{2k+3}\{K,R_{1},...,R_{p}\}$ has trivial stabiliser in $H/K$, so that an application of Kneser's theorem yields, \begin{equation*} |2^{2k+3}(C\cup \{0\})|=|K| |2^{2k+3}\{H,R_{1},...,R_{p}\}|\ge |K| \big(2^{2k+3}(p+1)-(2^{2k+3}-1)\big)\ge  2^{2k+3}p|K|\end{equation*} An application of the pigeon hole principle shows that, $|K|\ge |C|/(p+1)\ge |C|/2p$, it is now immediate that $|2^{2k+3}(C\cup\{0\})|\ge 2^{2k+2}|C|\ge 2^{k+2}n$.  We now obtain the required contradiction, as we have, \begin{equation*} \sum_{d\in 2^{2k+3}(C\cup\{0\})} \bar{\rho}(d)>2^{k+2}n\, \frac{|\bar{S}|}{2}  =2^{k+1}n|\bar{S}|\end{equation*} and so, using that $\bar{\rho}=\sum_{i} \rho_{i}$ and $|\bar{S}|=\sum_{i}|S_{i}|$, we must have for some $i$ that, \begin{equation*} \sum_{d\in 2^{2k+3}(C\cup\{0\})}\rho_{i}(d)>2^{k+1}n|S_{i}|\ge |S_{i}|^{2}\end{equation*} 
\\ \indent (iii) Suppose no coset is dense, now if some coset is not sparse we are done by (i), so we may assume all cosets are sparse.  If there are $2^{2k+5}$ or more cosets which are sparse but not very sparse, then we are done by (ii).  Hence we may assume there are at most $2^{2k+5}$ cosets which are sparse but not very sparse, since $|S\cap Q|\le 2^{k+1}n$ for all sparse cosets we may assume at most $2^{3k+6}$ elements of $S$ are in cosets that are sparse but not very sparse.  By (\ref{Sbig}) we have $|S|\ge 2^{3k+7}n$ so at least $2^{3k+6}n\ge 2n$ elements of $S$ belong to very sparse cosets.  Applying Lemma \ref{ssmall} to each very sparse coset, and averaging appropriately, we deduce that there exists $c\in C$ with \begin{equation*}\la(c)\ge \frac{1}{2}\sum_{Q}\defi_{Q}(S)=\frac{1}{2}\sum_{Q}|S\cap Q|\ge n\end{equation*} where the summation is taken over very sparse cosets $Q$, and so we are done.\end{proof}  

A very simple addition theorem in finite abelian groups is that if $A,B\subset G$ satisfy $|A|+|B|>|G|$ then $A+B=G$.  Let $\mathcal{D}$ denote the set of dense cosets, \begin{equation*} \mathcal{D}=\{Q\in G/H \,:\, Q \,\text{is dense}\} \, \subset G/H\end{equation*}  Let $\bar{\mathcal{D}}$ denote the set of all elements of dense cosets, \begin{equation*} \bar{\mathcal{D}}=\bigcup_{Q\in \mathcal{D}}Q \, \subset G \end{equation*}

\begin{claim} $S$ is not contained in $\bar{\mathcal{D}}$. \end{claim}

\begin{proof} We have from Theorem \ref{dvgms} that, $|\Sigma(C)|\ge |C|^{2}/64 \ge n^{2}/2^{2k+6}> 2^{k+1}n$.  Let $Q$ be a dense coset, then $|S\cap Q|\ge |H|-2^{k+1}n$, and so $|\Sigma(C)|+|S\cap Q|>|H|$ so by considering an appropriate shift of $S\cap Q$ we obtain (from the simple addition theorem stated above) that $Q\subset \Sigma(C)+(S\cap Q)\subset \Sigma (A)$.  Suppose the Claim is false, then every coset $Q$ meeting $S$ is dense, then it follows that $\Sigma(A)$ is a union of cosets of $H$ and so has non-trivial stabiliser, a contradiction.\end{proof}

If either of the conditions (i) or (iii) of Lemma \ref{la} hold then we are done, using this together with the above Claim, we note that we may proceed with the following assumptions,

 I\quad  \,\,\, Every coset is either sparse or dense\\ \indent
II \quad There is a dense coset\\ \indent 
III \quad \!\!\!\! $S$ is not contained in $\bar{\mathcal{D}}$

From this information we shall deduce that $|S|\ge \alpha_{k}n^{2}$, completing the proof.  To prove this it suffices by (\ref{SQbig}) to show that there are at least $2^{2k+1}$ dense cosets, ie. show $|\mathcal{D}|\ge 2^{2k+1}$.  We begin by recalling a Claim from \cite{dvgms}.

\begin{claim}\cite{dvgms} If $Q$ is dense and $b\in B$ then one of the cosets $Q+b$ or $Q-b$ is dense.\end{claim}

\begin{proof} If $b\in H$ then this is trivial.  So suppose $b\not \in H$, let $S_{-}$ be the set of elements in $S\cap Q$ which may be represented as a sum of elements of $B\setminus \{b\}$, and let $S_{+}$ be the set of elements in $S\cap Q$ which may be represented as $b$ plus a sum of elements of $B\setminus \{b\}$.  We note that $S\cap (Q+b)\supset S_{-}+b$ so that $|S\cap (Q+b)|\ge |S_{-}|$ and $S\cap (Q-b)\supset S_{+}-b$ so that $|S\cap (Q-b)|\ge |S_{+}|$.  As $|S_{-}|+|S_{+}|\ge |H|-2^{k+1}n> 2^{k+1}n$ we deduce that one of these sets and hence one of the sets $S\cap (Q+b),S\cap (Q-b)$ has cardinality greater than $2^{k+1}n$ the claim is then proved as we may assume all cosets are either sparse or dense.\end{proof}

Fix $Q_{0}\in \mathcal{D}$ and let $K$ be a maximal subgroup of $G/H$ for the property that $Q_{0}+K\subset \mathcal{D}$.  Let us also define a subgroup $\bar{K}$ of $G$ by, \begin{equation*} \bar{K}=\bigcup_{Q\in K}Q\end{equation*}

\begin{lemma} If $|B\cap \bar{K}|\ge 9|B|/10$ then $|S|\ge \alpha_{k}n^{2}$\end{lemma}

\begin{proof} We know that $B$ is not contained in $\bar{K}$, for if it were then we would have $S\subset \bar{K}\subset \bar{\mathcal{D}}$, contradicting III.  Let $b\in B\setminus \bar{K}$, we have from the above Claim that for each $Q\in Q_{0}+K$ that either $Q+b$ or $Q-b$ is dense.  This implies that there must be at least $|K|/2$ dense cosets outside of $Q_{0}+K$, so that the total number of dense cosets is at least $3|K|/2$.  Let us now estimate the size of $\bar{K}$.  Since $C\subset \bar{K}$ and $|B\cap \bar{K}|\ge 9|B|/10$ we deduce that $|A\cap \bar{K}|\ge 9|A|/10$.   Since $A'=A\cap \bar{K}$ is a subset of $A$, we have that $\Sigma(A')$ has trivial stabiliser, so that by applying the induction hypothesis of Theorem \ref{main} to $A'$ we obtain that $|\Sigma(A')|\ge 81\alpha_{k-1}n^{2}/100$, and since $\Sigma (A')\subset \bar{K}$ it follows that $|\bar{K}|\ge 81\alpha_{k-1}n^{2}/100$.  Since $|\bar{K}|=|K||H|$, it follows that $|H|\ge 81\alpha_{k-1}n^{2}/100|K|$.  We may assume that $|K|\le 2^{2k+1}$ since we are done if there are at least $2^{2k+1}$ dense cosets.  Let $Q$ be a dense coset then, \begin{equation*} |S\cap Q|\ge |H|-2^{k+1}n\ge \frac{81\alpha_{k-1}n^{2}}{100|K|}-2^{k+1}n\ge \frac{4\alpha_{k-1}n^{2}}{5|K|}\end{equation*} where the final inequality follows from $\alpha_{k-1}n/100|K|\ge n/2^{2k+15}\ge 2^{k}$.  Since there are at least $3|K|/2$ dense cosets $Q$ we have that, \begin{equation*} |S|\ge \frac{3|K|}{2}\frac{4\alpha_{k-1}n^{2}}{5|K|}=\frac{6\alpha_{k-1}n^{2}}{5}\ge \alpha_{k} n^{2}\vspace{-0.6cm} \end{equation*}\end{proof}

Hence we may assume that, \begin{equation}\label{out} |B\setminus \bar{K}|\ge \frac{|B|}{10}\ge \frac{n}{2^{k+4}}\ge 2^{4k+3}\end{equation}  Suppose that $B\setminus \bar{K}$ meets at least $2^{2k+2}$ distinct cosets of $H$, then it is possible to find cosets $Q_{1},...,Q_{2^{2k+1}}$ not in $\bar{K}$ which meet $B$ and have the property that $Q_{i}\neq -Q_{j}$ for each $i,j$.  From the Claim we have that for each $i=1,...,2^{2k+1}$ that one of the cosets $Q_{0}+Q_{i}$ or $Q_{0}-Q_{i}$ is dense.  Since we find a different dense coset for each $i=1,...,2^{2k+1}$ we find at least $2^{2k+1}$ dense cosets and we are done.

Hence we may assume that $B$ meets at most $2^{2k+2}$ cosets $Q$ of $H$ with $Q$ not in $\bar{K}$.  It follows from this, and (\ref{out}), that for one such coset $Q$ we have $|B\cap Q|\ge 2^{2k+1}$.  We write $\order(Q)$ for the order of $Q$ in $G/H$, the following lemma will allow us to deduce that $\order(Q)\ge 2^{2k+1}$.

\begin{lemma} Let $Q$ be a coset with $|B\cap Q|\ge \order(Q)$ and $\order(Q)\le 2^{2k+1}$, and let $R$ be a dense coset, then $R+Q$ is dense\end{lemma}

\begin{proof} Let $p=\order (Q)\le 2^{2k+1}$ and let $b_{1},...,b_{p-1}$ be $p-1$ elements of $B\cap Q$.  Since $R$ is dense we have $|S\cap R|\ge |H|-2^{k+1}n$.  We partition $S\cap R$ into, \begin{equation*}S^{+} = \big((b_{1}+...+b_{p-1})+\Sigma (B\setminus \{b_{1},...,b_{p-1}\}) \big)\cap R\end{equation*} the elements which can be expressed as a sum of elements of $B$ in such a way that all the elements $b_{1},...,b_{p-1}$ are included, and $S^{-}$ the elements of $S\cap R$ which can be expressed as a sum of elements of $B$, in which not all of the elements $b_{1},...,b_{p-1}$ are used.  We note that $-(b_{1}+...+b_{p-1})\in Q$ so we have that, $S\cap (R+Q)\supset S^{+}-(b_{1}+...+b_{p-1})$ and so $|S\cap (R+Q)|\ge |S^{+}|$.  We will also be able to relate $|S\cap (R+Q)|$ to $|S^{-}|$.  Consider the bipartite graph with vertex sets $S^{-}$ and $S\cap (R+Q)$ in which a vertex $x\in S^{-}$ is joined to $y\in S\cap (R+Q)$ if $y-x=b_{i}$ for some $i=1,...,p-1$; the vertices of $S^{-}$ all have positive degree, while the degree of vertices in $S\cap (R+Q)$ is certainly at most $p-1$, it follows that $|S\cap (R+Q)|\ge |S^{-}|/(p-1)$.  We now have $2|S\cap (R+Q)|\ge |S^{+}|+|S^{-}|/p\ge (|S^{+}|+|S^{-}|)/p\ge |S\cap R|/p$, recall that $R$ is dense so that by (\ref{SQbig}) we have, $|S\cap R|\ge \alpha_{k}n^{2}/2^{2k+1}\ge n^{2}/2^{2k+7}$, using this together with the fact $p\le 2^{2k+1}$ we obtain, \begin{equation*} |S\cap (R+Q)|\ge \frac{|S\cap R|}{2p}\ge \frac{n^{2}}{2^{4k+9}}>2^{k+1}n \end{equation*} the final inequality follows from the bound $n\ge n_{k}\ge 2^{5k+10}$.  This shows us that $R+Q$ is not sparse, since we are assuming that all cosets are either sparse or dense it follows that $R+Q$ is dense.\end{proof}

So if it were the case that $\order(Q)\le 2^{2k+1}$ then the above lemma would show that $Q$ is in the stabiliser of $\mathcal{D}$, but now the subgroup $\langle K\cup \{Q\}\rangle$ contradicts the maximality of $K$.  Hence we may assume $\order(Q)\ge 2^{2k+1}$, our proof of the theorem is completed with this final lemma.

\begin{lemma} Suppose there is a coset $Q$ with $|B\cap Q|\ge 2^{2k+1}$ and $\order(Q)\ge 2^{2k+1}$, then there are at least $2^{2k+1}$ dense cosets.\end{lemma}

\begin{proof} We know there is at least one dense coset $R$.  Consider the cosets $R-iQ$ for positive integers $i$, if they are all dense then we have found $\order (Q)\ge 2^{2k+1}$ dense cosets and we are done, hence we may assume there is some non-negative integer $i$ such that $R-iQ$ is dense, but $R-(i+1)Q$ is not.  Set $Q_{0}=R-(i+1)Q$ and $Q_{1}=R-iQ$, and for $j=2,...,2^{2k+1}$ let $Q_{j}=Q_{1}+(j-1)Q$.  We show for each $j=2,...,2^{2k+1}$ that $Q_{j}$ is dense.  Let $b_{1},...,b_{2^{2k+1}}$ be $2^{2k+1}$ elements of $B\cap Q$.  We partition of $S\cap Q_{1}$ into $S^{-}=\Sigma(B\setminus \{b_{1},...,b_{2^{2k+1}}\})$, the set of elements which may be expressed as a sum of elements of $B$ without using any of the elements $b_{1},...,b_{2^{2k+1}}$, and $S^{+}$ the elements of $S\cap Q_{1}$ which can be expressed as a sum of elements of $B$, in such a way that at least one of the elements $b_{1},...,b_{2^{2k+1}}$ is used.  We can relate the size of $S^{+}$ to the size of $S\cap Q_{0}$ and this will allow us to bound the size of $S^{+}$.  Consider the bipartite graph with vertex sets $S^{+}$ and $S\cap Q_{0}$ in which a vertex $x\in S^{+}$ is joined to $y\in S\cap Q_{0}$ if $x-y=b_{i}$ for some $i=1,...,2^{2k+1}$; the vertices of $S^{+}$ all have positive degree, while the degree of vertices in $S\cap Q_{0}$ is certainly at most $2^{2k+1}$, it follows that $|S^{+}|\le 2^{2k+1}|S\cap Q_{0}|$.  However we have that $Q_{0}$ is not dense, since we are assuming all cosets are either sparse or dense we may assume $Q_{0}$ is sparse, so that $|S\cap Q_{0}|\le 2^{k+1}n$, it follows that $|S^{+}|\le 2^{3k+2}n$.  As $Q_{1}$ is dense we have a lower bound on $|S\cap Q_{1}|$ from (\ref{SQbig}), noting that $S\cap Q_{1}=S^{+}\cup S^{-}$ we obtain,\begin{equation*} |S^{-}|\ge |S\cap Q_{1}|-|S^{+}|\ge \frac{\alpha_{k} n^{2}}{2^{2k+1}}-2^{3k+2}n\ge \frac{n^{2}}{2^{2k+7}}-2^{3k+2}n >2^{k+1}n  \end{equation*} the final inequality following from the fact $n\ge n_{k}> 2^{5k+10}$.  Now for each $i=2,...,2^{2k+1}$ we have $S\cap Q_{i}\supset S^{-}+b_{1}+...+b_{i-1}$ and so for each $i=2,...,2^{2k+1}$ we have $|S\cap Q_{i}|>2^{k+1}n$.  This implies that these cosets are not sparse, so they are dense.\end{proof}

\section{Proof of Theorem \ref{asymptoticolson}}\label{poao}

In this section we prove Theorem \ref{asymptoticolson}.  The first half of the section is devoted to proving Lemma \ref{thegoodincreaseolson}, which is the appropriate variant of Lemma \ref{thegoodincrease} for the new setting of $A\subset \mathbb{Z}_{n}^{*}\subset \mathbb{Z}_{n}$.  We then turn in the second half of the section to applying this, and previous lemmas, to prove the required result.  As in Section \ref{prel}, it is important to find lower bounds on the cardinalities of the sets $rC^{*}$, where $C^{*}=C\cup (-C)\cup \{0\}$.

\begin{lemma}\label{rcolson} Let $C\subset H$ be such that $H=\langle \{c\}\rangle$ for each $c\in C$ and $C\cap (-C)=\phi$, then for all positive integers $r$ we have that either $rC^{*}=H$ or, \begin{equation*} |rC^{*}|\ge 2r|C| \end{equation*}\end{lemma}

\begin{proof} Let $K$ be the stabiliser of $rC^{*}$.  If there is an element $c$ in $C\cap K$ then $K=H$ and $rC^{*}=H$ and we are done.  So we may assume $C^{*}\cap K=\{0\}$.  Let $Q_{1},...,Q_{p}$ be the non-trivial cosets of $K$ meeting $C^{*}$, it follow that $rC^{*}$ is the set of all elements of the cosets in $r\{K,Q_{1},...,Q_{p}\}$.  By the definition of $K$ the set $r\{K,Q_{1},...,Q_{p}\}$ has trivial stabiliser in $H/K$ so that an application of Kneser's theorem yields, \begin{equation*} |r\{K,Q_{1},...,Q_{p}\}|\ge r(p+1)-(r-1)\ge rp\end{equation*} and so $|rC^{*}|\ge rp|K|$.  So to complete the proof we must show that,\begin{equation*} |K|\ge \frac{2|C|}{p} \end{equation*}  
However we have that the set $C'=C\cup (-C)$ has size $2|C|$ and is contained in $\bigcup_{i=1}^{p} Q_{i}$, so that we may deduce from the pigeon hole principle that for some coset $Q$ we have, \begin{equation*} |C'\cap Q|\ge \frac{|C'|}{p}= \frac{2|C|}{p}\end{equation*} and so noting $|K|=|Q|\ge |C'\cap Q|$ we are done.\end{proof}

We now deduce the key result we need for finding elements $c\in C$ with large $\la(c)$.

\begin{lemma}\label{thegoodincreaseolson} Let $C\subset H$ be such that $H=\langle \{c\} \rangle$ for each $c\in C$ and $C\cap (-C)=\phi$.  Let $\gamma > 2$ and let $S\subset H$ be such that $\defi_{H}(S)\ge \gamma |C|$ then there exists $c\in C$ with $\la (c)\ge (1-4\gamma^{-1})|C|$\end{lemma}

\begin{proof} For all $d\in H$ we have $\la_{S}(d)=\la_{H\setminus S}(d)$, so that exchanging $S$ with $H\setminus S$ if required, we may assume $\defi_{H}(S)=|S|\le |H|/2$.  As in the proofs of Section \ref{prel} we shall proceed by assuming the result fails and from this deduce that $\sum_{d} \rho(d)> |S|^{2}$, a contradiction.  It will be useful to use the fact that $\sum_{d}\rho(d)=\int_{0}^{|S|}|D_{t}| \, dt $ where $D_{t}=\{d\in H:\rho(d)\ge t\}$.\\ \indent Suppose that every element $c\in C$ has $\la(c)<(1-4\gamma^{-1})|C|$, this implies $\la(c)<(1-4\gamma^{-1})|C|$ for all $c\in C^{*}$, since $\la(0)=0$ and $\la(-c)=\la(c)$.  So that by the subadditivity of $\la$ we have for all $r$ and for all $d\in rC^{*}$ that $\la (d)<r(1-4\gamma^{-1})|C|$, and so $\rho(d)>|S|-r(1-4\gamma^{-1})|C|$. 

For each $r=1,...,\lfloor |S|/|C|\rfloor$, we define, \begin{equation*} I_{r}=(|S|-(r+1)(1-4\gamma^{-1})|C|,|S|-r(1-4\gamma^{-1})|C|]\end{equation*} For $t\in I_{r}$ we have $t\le |S|-r(1-4\gamma^{-1})|C|$ and so $D_{t}\supset rC^{*}$, applying Lemma \ref{rcolson} we deduce, $|D_{t}|\ge 2r|C|$. 
\\ \indent Note also that for $t\le |S|-(\lfloor |S|/|C| \rfloor+1) (1-4\gamma^{-1})|C|$ we have $D_{t}\supset \lfloor |S|/|C|\rfloor C^{*}$ and so $|D_{t}|\ge 2 \lfloor |S|/|C|\rfloor |C|$, so using the fact that $\lfloor |S|/|C|\rfloor\ge (|S|-|C|)/|C|\ge(1-\gamma^{-1})|S|/|C|$ we have, \begin{equation*} |D_{t}|\ge 2\Big\lfloor \frac{|S|}{|C|}\Big\rfloor |C|\ge 2(1- \gamma^{-1}) |S| \end{equation*} for all $t\le |S|-(\lfloor |S|/|C| \rfloor +1) (1-4\gamma^{-1})|C|$.    Note that $\lfloor |S|/|C|\rfloor +1\le |S|/|C|+1\le (1+\gamma^{-1})|S|/|C|$ so that the above bound on $|D_{t}|$ hold for all $t$ satisfying $t\le |S|-(1-4\gamma^{-1})(1+\gamma^{-1})|S|$ and so certainly for all $t\le 3\gamma^{-1}|S|$.  We obtain, \begin{equation*}  \sum_{d\in H}\rho(d)= \int_{0}^{|S|}|D_{t}|\, dt\ge \sum_{r=1}^{\lfloor |S|/|C|\rfloor}\int_{t\in I_{r}}|D_{t}|\, dt+\int_{0}^{3\gamma^{-1}|S|}|D_{t}|\, dt \end{equation*} 
and so, \begin{equation*} \sum_{d\in H}\rho(d) \ge \sum_{r=1}^{\lfloor |S|/|C|\rfloor} 2(1-4\gamma^{-1})r|C|^{2}+6\gamma^{-1}(1-\gamma^{-1})|S|^{2}\end{equation*}  
Since $\lfloor |S|/|C|\rfloor \ge (1-\gamma^{-1})|S|/|C|$ we have that $\sum_{r=1}^{\lfloor |S|/|C|\rfloor }r= (\lfloor|S|/|C|\rfloor)(\lfloor|S|/|C|\rfloor+1)/2 \ge (1-\gamma^{-1})|S|^{2}/2|C|^{2}$, and so, \begin{equation*}  \sum_{d\in H}\rho(d)\ge (1-\gamma^{-1})  (1-4\gamma^{-1}) |S|^{2} + 6\gamma^{-1}(1-\gamma^{-1})|S|^{2}\end{equation*}  However this quantity is larger than $|S|^{2}$, a contradiction. \end{proof}

\subsection*{Proof of Theorem \ref{asymptoticolson}}

Since the required result is asymptotic it suffices to prove it for $n\ge n_{0}$, for some $n_{0}$.  Let $n_{0}$ be chosen such that $\sqrt{n}\ge 160n^{1/4}$ and $\log_{3/2}(n)\le n^{1/4}$ for all $n\ge n_{0}$.  For $n\ge n_{0}$ we define $f_{4}(n)=1+15n^{-1/4}$ and $f_{3}(n)=2(f_{4}(n)+n^{-1/2})=2+30n^{-1/4}+2n^{-1/2}$.  It is clear that $f_{3}(n)$ is of the form $2+o(1)$.  To prove Theorem \ref{asymptoticolson} we show that a set $A\subset \mathbb{Z}_{n}^{*}\subset \mathbb{Z}_{n}$ with $|A|\ge f_{3}(n)\sqrt{n}$ must have $\Sigma(A)=\mathbb{Z}_{n}$.  Let us note that if $|A|\ge f_{3}(n)\sqrt{n}=2f_{4}(n)\sqrt{n}+2$ then it is possible to partition $A$ into subsets $A_{1}$ and $A_{2}$, each with cardinality at least $f_{4}(n)\sqrt{n}$ and satisfying $A_{i}\cap (-A_{i})=\phi$ for $i=1,2$.  Our proof will work by showing that, \begin{equation*} |\Sigma(A_{i})|>n/2 \qquad \text{for}\,\, i=1,2\end{equation*} We are then done as $\Sigma(A)\supset \Sigma(A_{1})+\Sigma(A_{2})$ and $S+T\supset \mathbb{Z}_{n}$ whenever $S,T\subset \mathbb{Z}_{n}$ are such that $|S|+|T|>n$.

During the proof we will rely essentially on the results we have proved which give us an element $c\in C$ with large value of $\la(c)$ we recall now the three bounds which we shall need for the sequel.  We write $\la(S,C)$ for the maximum value of $\la_{S}(c)$ over elements $c\in C$.
\begin{eqnarray*} \la(S,C)\ge \frac{|S|}{2} \qquad\text{for}\, \defi_{\mathbb{Z}_{n}}(S)\le |C|/2 \\ \\ \la(S,C)\ge \frac{|C|}{8} \qquad\text{for}\, \defi_{\mathbb{Z}_{n}}(S)\ge |C|/2 \\  \la(S,C)\ge (1-4n^{-1/4})|C| \qquad\text{for}\, \defi_{\mathbb{Z}_{n}}(S)\ge n^{1/4}|C| \end{eqnarray*}  These bounds are taken from Lemmas \ref{ssmall}, \ref{slarge} and \ref{thegoodincreaseolson} respectively.  It is valid to apply these lemmas for sets $C\subset A_{i}$, $i=1,2$, because this certainly implies $C\cap (-C)=\phi$ and that $C\subset\mathbb{Z}_{n}^{*}$ and so $\mathbb{Z}_{n}$ is generated by each element $c\in C$.

Fix $i\in\{1,2\}$, we show that $|\Sigma(A_{i})|>n/2$.  We show this by building up a set $B\subset A_{i}$ with $\Sigma(B)$ large.  We define, \begin{equation*} g(t)=\max_{B\subset A_{i}, \, |B|=t} |\Sigma(B)|\end{equation*} Given $B\subset A_{i}$ of cardinality $t$, and such that $|\Sigma(B)|=g(t)$, we let $S=\Sigma(B)$, and $C=A_{i}\setminus B$, let $\la(S,C)=\max_{c\in C}\la_{S}(c)$.  By considering $B\cup\{c\}$ a set of cardinality $t+1$ we deduce that, \begin{equation*}g(t+1)\ge g(t)+\la(S,C)\end{equation*}  Note that we may assume at all times that $|S|\le n/2$ so that $\defi_{\mathbb{Z}_{n}}(S)=|S|$, for we are immediately done if $|S|>n/2$, so we use the three inequalities above concerning $\la(S,C)$, with $|S|$ in the place of $\defi_{\mathbb{Z}_{n}}(S)$.  

\begin{lemma}\label{asolson} Let $n\ge n_{0}$.  Let $A\subset \mathbb{Z}_{n}^{*}\subset \mathbb{Z}_{n}$ with $A\cap (-A)=\phi$ and $|A|\ge \sqrt{n}+15n^{1/4}$, then $|\Sigma(A)|>n/2$.\end{lemma}

\begin{proof} It suffices to prove this when the size of $A$ is the integer above $\sqrt{n}+15n^{1/4}$, so we may assume $|A|\le 9\sqrt{n}/8$.  We put lower bounds on $g(t)$ using the bounds on $\la(S,C)$ given above.  We work in three main stages, corresponding to the three different lower bounds we have on $\la(S,C)$.  
\\ \indent \textbf{Stage 1:} Let $t_{1}$ be the least $t$ such that $g(t)\ge (|A|-t)/2$.  We note that $g(1)=2$.  Let $2\le t <t_{1}$, let $B$ be a set of size $t$ with $S=\Sigma(B)$ satisfying $|S|=g(t)$ and let $C=A\setminus B$, we have $|S|\le (|A|-t)/2=|C|/2$.  From our bounds on $\la$ we have, $\la(S,C)\ge |S|/2$ and so, $g(t+1)\ge 3g(t)/2$ so that $g(t)\ge 2(3/2)^{t-1}\ge (3/2)^{t}$.  It follows that, $t_{1}\le \log_{3/2} (|A|)\le \log_{3/2}(n)\le n^{1/4}$. 
\\ \indent \textbf{Stage 2:} Let $t_{2}$ be the least $t$ such that $g(t)\ge 9n^{3/4}/8$.  Let $t_{1}\le t\le \min\{t_{1}+9n^{1/4},t_{2}\}$, let $B$ be a set of size $t$ with $S=\Sigma(B)$ satisfying $|S|=g(t)$, and let $C=A\setminus B$, we have that $|S|\ge |C|/2$ and $|C|=|A|-t \ge \sqrt{n}$, so that $\la(S,C)\ge |C|/8\ge \sqrt{n}/8$ and so, $g(t+1)\ge g(t)+\sqrt{n}/8$.  So that $g(t)\ge g(t_{1})+(t-t_{1})\sqrt{n}/8$ and so certainly $t_{2}\le t_{1}+9n^{1/4}\le 10n^{1/4}$ 
\\ \indent \textbf{Stage 3:} Let $t_{3}$ be the least $t$ such that $g(t)>n/2$.  Let $t_{2}\le t<t_{3}$ and let $B$ be a set of size $t$ with $S=\Sigma(B)$ satisfying $|S|=g(t)$, and let $C=A\setminus B$.  We have that $|S|=g(t)\ge g(t_{2})\ge 9n^{3/4}/8\ge n^{1/4}|C|$ (certainly we have that $|C|\le |A|\le 9\sqrt{n}/8$), so that we have, \begin{equation*} \la(S,C)\ge (1-4n^{-1/4})|C|=(1-4n^{1/4})(|A|-t)\end{equation*} which gives us that, $g(t+1)\ge g(t)+(1-4n^{-1/4})(|A|-t)$ for $t_{2}\le t <t_{3}$, we now use these to calculate $g(t)$ in this range, \begin{equation*} g(t)\ge g(t_{2})+(1-4n^{-1/4})\sum_{t'=t_{2}}^{t}(|A|-t')=g(t_{2})+(1-4n^{-1/4})\sum_{t'=|A|-t}^{|A|-t_{2}}t'\end{equation*} using the bound $\sum_{t'=r}^{s}t'\ge (s^{2}-r^{2})/2=(s+r)(s-r)/2\ge s(s-r)/2$ and the bound $|A|-t_{2}\ge \sqrt{n}+5n^{1/4}$ we obtain, \begin{equation*} g(t)\ge (1-4n^{-1/4})\frac{(|A|-t_{2})(t-t_{2})}{2}\ge (1-4n^{-1/4})\frac{(\sqrt{n}+5n^{1/4})(t-t_{2})}{2}> \frac{\sqrt{n}}{2} (t-t_{2})\end{equation*} it is immediate then that $t_{3}-t_{2}\le \sqrt{n}$, whence $t_{3}\le \sqrt{n}+10n^{1/4}$.  Of course we now have, $|\Sigma(A)|\ge g(t_{3})>n/2$.\end{proof}

A more delicate version of the above proof, in which there are $\log{n}$ stages (an initial stage as Stage 1 above, followed by $\log{n}$ stages in which $g(t)$ doubles) allows one to deduce Lemma \ref{asolson} with the condition on $|A|$ weakened to $|A|\ge \sqrt{n}+10\log_{2}{n}$.  This then implies that any subset $A\subset \mathbb{Z}_{n}^{*}\subset \mathbb{Z}_{n}$ with $|A|\ge 2\sqrt{n}+20\log_{2}{n}+2$ must have $\Sigma(A)=\mathbb{Z}_{n}$.

\end{document}